\documentclass[12pt]{amsart}

\usepackage{latexsym,amsmath,amssymb,amsfonts,amscd,graphics}
\usepackage[all,ps,cmtip]{xy}

\usepackage{tikz-cd}
\usepackage{scrextend}
\usepackage{hyperref}
\usepackage[margin=3cm]{geometry}

\numberwithin{equation}{section}
\theoremstyle{plain}

\newtheorem{theorem}[equation]{Theorem}
\newtheorem*{mainthm}{Main Theorem}

\newtheorem{lemma}[equation]{Lemma}
\newtheorem{prop}[equation]{Proposition}
\newtheorem{cor}[equation]{Corollary}

\theoremstyle{remark} 
\newtheorem{remark}[equation]{Remark}

\theoremstyle{defn}
\newtheorem{defn}[equation]{Definition}

\newcommand\addtag{\refstepcounter{equation}\tag{\theequation}}
\newcommand{\bP}{\mathbb{P}}
\newcommand{\bA}{\mathbb{A}}
\newcommand{\bR}{\mathbb{R}}
\newcommand{\bN}{\mathbb{N}}
\newcommand{\bZ}{\mathbb{Z}}
\newcommand{\bC}{\mathbb{C}}
\newcommand{\calA}{\mathcal{A}}
\newcommand{\calB}{\mathcal{B}}
\newcommand{\calO}{\mathcal{O}}
\newcommand{\calC}{\mathcal{C}}
\newcommand{\calM}{\mathcal{M}}
\newcommand{\calN}{\mathcal{N}}
\newcommand{\calX}{\mathcal{X}}
\newcommand{\calD}{\mathcal{D}}
\newcommand{\calU}{\mathcal{U}}
\newcommand{\calL}{\mathcal{L}}

\newcommand{\calV}{\mathcal{V}}
\newcommand{\Hom}{\mathrm{Hom}}
\newcommand{\Spec}{\mathrm{Spec}}
\newcommand{\Al}{AB(V)}
\newcommand{\Alm}{AB(V)_{\textrm{main}}}
\newcommand{\kg}{{\calA\calB}(V)}
\newcommand{\Chow}{[V \sslash_C H]}

\usepackage{graphicx}
\usepackage{tikz}
\tikzset{node distance=2cm, auto}
\newcommand{\sslash}{\mathbin{/\mkern-6mu/}}

\begin{document}

\title[]{Logarithmic Stable Toric Varieties and Their Moduli}
\author{Kenneth Ascher}
\email{kenascher@math.brown.edu}
\address{Mathematics Department, Brown University, 151 Thayer Street, Providence, RI 02912}
\author{Samouil Molcho}
\email{Samouil.Molcho@colorado.edu}
\address{University of Colorado, Boulder, Boulder, Colorado 80309-0395, USA}
\thanks{Research of Ascher and Molcho supported in part by funds from NSF grant DMS-1162367.}
\begin{abstract} The Chow quotient of a toric variety by a subtorus, as defined by Kapranov-Sturmfels-Zelevinsky, coarsely represents the main component of the moduli space of stable toric varieties with a map to a fixed projective toric variety, as constructed by Alexeev and Brion. We show that, after endowing both spaces with the structure of a logarithmic stack, the spaces are isomorphic. Along the way, we construct the Chow quotient stack and demonstrate several properties it satisfies.   \end{abstract}

\maketitle

\section{Introduction}
We work over an algebraically closed field of characteristic 0. \\

Alexeev and Brion construct a moduli stack $AB(V)$  parametrizing finite torus equivariant maps from stable (also known as broken) toric varieties $X$ to a fixed projective variety $V$  \cite{ab}, generalizing Alexeev's construction of the moduli space of stable toric pairs: pairs $(X, B)$ where $X$ is a stable toric variety and $B$ is a divisor satisfying certain numerical and singularity criteria \cite{ale}. As is common with moduli spaces of higher dimensional varieties, the modular compactification has several irreducible components. Adding a logarithmic structure often gives one hope towards isolating the main component of a compactification. In \cite{Olsson}, Olsson introduces logarithmic geometry as a method for compactifying moduli spaces of abelian varieties. Along the way, he enriches Alexeev's moduli space of stable toric pairs with a logarithmic structure and shows that this stack carves out the main component of Alexeev's space of stable toric pairs.

In this paper we will give the space of maps $\Al$ an analogous logarithmic structure, $\kg$ in the case where the target variety is toric, and show that $\kg$ is isomorphic  as logarithmic stacks to $[V \sslash_C H]$, the Chow quotient of a toric variety by a subtorus as defined in \cite{ksz} endowed with a logarithmic stack structure.  Not only does this construction give an explicit description of a stack of logarithmic stable maps, but it also isolates the main component of $\Al$, just as in the case of stable toric pairs.
\begin{mainthm} Let $V$ be a projective toric variety and let $H \subset T$ be a subtorus. The stack $\kg$ parametrizing logarithmic maps of stable toric varieties to $V$ is isomorphic to the Chow stack $[V \sslash_C H]$. In particular, $\kg$ is a logarithmically smooth, proper, and irreducible algebraic stack with finite diagonal. Moreover, $\kg$ is isomorphic to the normalization of the main component of Alexeev-Brion's space, $\Alm$. \end{mainthm}

Along the way, we construct the Chow quotient stack, a moduli stack parametrizing orbits of the action of a subtorus on a toric variety, and prove several properties it enjoys. In particular, the following universal property of the Chow quotient stack is a crucial element of the proof of our main theorem:

\begin{theorem}
Let $V$ be a projective toric variety. The Chow quotient stack is the terminal object of $\calC_V$, the category of toric families of toric $H$-stacks mapping to $V$. \end{theorem}

Furthermore, it carries a universal family:

\begin{theorem} Fibers of the universal morphism $\calU \to \Chow$ are connected stable toric varieties. \end{theorem}

Finally, in a similar vein to Remark 1.21 \cite{gs}, we prove (c.f. Theorem 3.19):

\begin{theorem}The dual cone of the basic monoid is the moduli space of tropical broken toric varieties of the given type. \end{theorem}

We will now sketch some previous results for the one-dimensional case where the subtorus $H = \bC^*$.

\subsection{Previous work}

In \cite{ksz},  the Chow quotient $V \sslash_C  H$ is constructed, where $V$ is a toric variety
and $H$, thought of as a $k$-parameter subgroup, is a subtorus of the torus $T$. In fact, it is also shown that $V \sslash_C H$ is a toric variety \footnote{In this paper we adopt the convention that the Chow quotient is a normal toric variety.} and the fan structure is described purely using combinatorial data from the fan of $V$ and the subtorus $H$. In \cite{cs}, it was shown that if one considers $H = \bC^*$ to be a one-parameter subgroup, then the Chow quotient $V \sslash_C \bC^*$ coincides with the coarse moduli space of $K_{\Gamma}(V)$, the moduli space of logarithmic stable maps constructed independently by Chen \cite{chen} (in the rank one case),  Abramovich-Chen \cite{ac} and Gross-Siebert \cite{cs}. To show this isomorphism, the authors endow $V$ with its natural logarithmic structure as a toric variety and use the one-parameter subgroup $H$ to determine the discrete data $\Gamma$ specifying the stack structure on $K_{\Gamma}(V)$. This gives a somewhat different, explicit description of the moduli space of logarithmic stable maps.

Finally, in \cite{sam}, an additional level of structure is added to the Chow quotient, making $[V \sslash_C \bC^*]$ into a toric stack. One can obtain this structure explicitly by enriching the fan of $V \sslash_C \bC^*$ by adding natural combinatorial data that arises from $V$ and $H$. This toric stack is then shown to be isomorphic to the logarithmic stack $K_{\Gamma}(V)$, giving very explicit descriptions of several stacks of logarithmic stable maps. Thus, this paper serves to find a higher dimensional analogue of $K_{\Gamma}(V)$, generalizing the result where $H$ is a one-parameter subgroup.

We will begin by reviewing some of the techniques involved in the above mentioned constructions. We will, however, assume knowledge in logarithmic geometry. For background, we refer the reader to the survey \cite{hom}.

\subsection*{Outline} 
\begin{addmargin}[42pt]{0\linewidth}
\begin{itemize}
\item[Sec. 2 (Pg. \pageref{2})] Definition and construction of a toric stack, discussion of various notions of toric stacks appearing in the literature, and definition of a logarithmic stack
\item[Sec. 3 (Pg. \pageref{3})] Discussion of the KSZ construction of the Chow quotient variety, construction of the Chow quotient as a toric stack, and discussion of various properties of the Chow quotient stack
\item[Sec. 4 (Pg \pageref{4})] Definition of stable toric varieties, discussion of logarithmic structure and definition of the logarithmic stack parametrizing logarithmic stable maps of stable toric varieties
\item[Sec. 5 (Pg \pageref{5})] Proof of main theorem describing the equivalence of the two logarithmic stacks
\end{itemize}
\end{addmargin}

\subsection*{Notation} We will use $X$ to denote a stable toric variety in the sense of Alexeev and $V$ to describe its target -- a fixed projective toric variety. When we wish to discuss the scheme underlying a logarithmic scheme, we will use $\underline{X}$.  The notation $V \sslash_C H$ will denote the Chow quotient and $[V \sslash_C H]$ will represent the Chow quotient stack. Finally, we will use $\Al$ to denote Alexeev and Brion's stack of stable maps of toric varieties to $V$ and $\kg$ will denote the stack of logarithmic stable maps of toric varieties.

\section{Toric Stacks} \label{2}

There are several definitions of toric stacks appearing in the literature. Borisov-Chen-Smith first introduced toric stacks in \cite{bcs}, where they construct smooth Deligne-Mumford stacks arising from simplicial fans. The notion which we will primarily use in this paper, was first introduced by Tyomkin \cite{tyo} as a generalization of Borisov-Chen-Smith's construction. The work of Geraschenko and Satriano (see \cite{gs1} and \cite{gs2}) provides an extensive  description of several notions of toric stacks. In addition, they develop a theory of toric stacks that encompasses, among other theories, the work of both Borisov-Chen-Smith as well as Tyomkin. The main purposes of this section are to introduce toric stacks, and to show that the notion we adopt is present in the theory of Geraschenko-Satriano.

Geraschenko-Satriano define a toric stack to be the stack quotient of a normal toric variety $X$ by a subgroup $G$ of its torus $T_0$. In this case, the stack $[X/G]$ has a dense open torus $T = T_0/G$ and so a toric stack can be thought of as a stack with an action of a dense torus. The following definition which we will use, is adopted from  Tyomkin's construction of toric stacks, see Definition 4.1 in \cite{tyo} for more details.

\begin{defn} Toric stack datum is a triple $(F, N_{\sigma}, N)$ where $F$ is a fan in a lattice $N$ and for each cone $\sigma \in F$ we specify a finitely generated and saturated submonoid $N_{\sigma} \subset \sigma \cap N$. We require that if $\tau < \sigma$ is a face, then $N_{\tau} = \tau \cap N_{\sigma}$. We also require that for all maximal cones $\sigma$, the group $N^{\textrm{gp}}_{\sigma} \subset N$ has finite index. \end{defn}

We can realize a toric stack geometrically from this datum as follows. Since the inclusions $N^{\textrm{gp}}_{\sigma} \subset N$ have finite index,  they induce maps of tori $T(N^{\textrm{gp}}_{\sigma}) \to T(N)$ whose kernels $K_{\sigma}$ are finite subgroups of $N_{\sigma}^{\textrm{gp}}$.  Let $X_{\sigma}$ be the toric variety associated to the cone $\sigma$ in the lattice $N^{gp}_{\sigma}$. We then define $\calX_{\sigma}$ to be the stack quotient $[X_{\sigma} / K_{\sigma}] $ and thus the compatibility condition in the toric stack datum allows one to glue the $\calX_{\sigma}$ together to obtain a toric stack $\calX(F, N_{\sigma}, N)$. 

\begin{remark} In this paper, all maximal cones will be of full dimension. We do note, however, that the definition and realization above generalizes to situations to where this is not the case. We will not discuss this in this paper. \end{remark}

\begin{defn} A morphism of toric stack data: $(F, N_{\sigma}, N) \to (G, M_{\tau}, M)$ is a morphism of lattices $N \to M$ that takes every cone $\sigma \in F$ into a cone of $G$. We also require that if $\sigma$ maps into $\tau$, then $N_{\sigma}$ maps to $M_{\tau}$. \end{defn}

In fact a toric stack will naturally be a \textit{logarithmic stack}. Following Shentu \cite{shentu},  we will use the following definition for a logarithmic algebraic stack (see also Chapter 5 of \cite{olssonlog} for a discussion on various notions of algebraic stacks).

\begin{defn}\cite[Definition 4.2]{shentu}
Given an algebraic stack $\calX$  define a logarithmic structure on $\calX$, as a pair $(\calM, \alpha)$, where $\calM$ is a sheaf of monoids, and $\alpha: \calM \to \calO_{\calX}$ is a homomorphism of monoids, such that $\alpha|_{\alpha^{-1}\calO_{\calX}^*}$ is an isomorphism. We call the pair $(\calM, \alpha)$ a logarithmic algebraic stack. \end{defn}

\begin{remark} The stack $\calX(F, N_{\sigma}, N)$ is a logarithmic stack.  The charts for the logarithmic structure of the affine pieces $\calX_{\sigma}$ are given by the dual monoid $M_{\sigma} = \Hom(N_{\sigma}, N)$. \end{remark}

\begin{remark} The coarse moduli space of $\calX(F, N_{\sigma}, N)$ is the toric variety $X(F)$ associated to the fan $F$. \end{remark}
\subsection{Alternate defintions}

We will now give a more careful definition of a Geraschenko-Satriano toric stack, before stating the main result of \cite{gs2}, which gives conditions on when an Artin stack is toric.

Let $(F, \beta: L \to N)$ be a stacky fan, i.e. a pair of a fan $F$ in a lattice $L$ and a map $\beta$ to a lattice $N$ with finite cokernel. Then $\beta^*$ induces a surjective map of tori, $T_{\beta} : T_L \to T_N$, with kernel $G_{\beta}$. 

\begin{defn} (GS toric stack) Let $(F, \beta: L \to N)$ be a pair as above. Then a GS toric stack $\calX_{F, \beta}$ is defined to be the stack quotient $[X(F) / G_{\beta}]$. \end{defn} 

For the sake of brevity, when referring to \textit{toric stacks} we mean the toric stacks arising from the data in Definition 2.1. We will refer to \textit{GS toric stacks} to indicate the toric stacks of Geraschenko-Satriano.

As mentioned above, in \cite{gs2} the authors develop a criterion for showing that an Artin stack is a GS toric stack. The criteria has been somewhat relaxed due to work of Alper, Hall, and Rydh (see Remark 4.4.0 in \cite{gs2}). We recall the necessary criteria below:

\begin{theorem} [Theorem 6.1 \cite{gs2}] \label{gs} Let $\calX$ be an Artin stack of finite type over an algebraically closed field $k$ of characteristic 0. Suppose $\calX$ has an action of a torus $T$ and a dense open substack which is $T$-equivariantly isomorphic to $T$. Then $\calX$ is a  GS toric stack if and only if the following hold:
\begin{enumerate}
\item $\calX$ is normal
\item $\calX$ has affine diagonal and
\item geometric points of $\calX$ have linearly reductive stabilizers 
\end{enumerate}
\end{theorem}

Using this theorem, Geraschenko-Satriano show the following remark:

\begin{remark} \label{2stacks} \cite[Remark 6.2]{gs2} The toric stacks defined by Tyomkin arising from the data in Definition 2.1 are GS toric stacks. \end{remark} 

Finally, we discuss a property of maps between toric stacks that will be used to conclude our main result in the final section of this paper.

\begin{lemma} \label{keylemma} Let $\calX_F = (F,N_{\sigma},N)$ and $\calX_G = (G, M_{\tau}, M)$ be two toric stacks and suppose $M = N$. Furthermore, suppose there are morphisms in both directions between $\calX_F$ and $\calX_G$. Then $\calX_F \cong \calX_G$. \end{lemma}

\begin{proof} In one direction, a cone $\sigma \in F$ must be a subcone of a cone $\tau \in G$. The morphism in the other direction tells us that $\tau$ is a subcone of $\sigma$. So $F = G$. Additionally, for each cone $\sigma \in F$, the submonoid $N_{\sigma}$ must be a submonoid of $M_{\sigma}$ and vice versa, so $N_{\sigma} = M_{\sigma}$. As the lattices, fans, and submonoids are all the same, the two toric stacks must be equivalent. \end{proof}

\section{Chow Quotients of Toric Varieties} \label{3}

\subsection{Construction of the quotient as a variety}

We will begin by recalling the definition and construction of the Chow quotient of a toric variety by a subtorus, following \cite{ksz}. 

Let $V$ be a projective toric variety with big torus $T$ embedded as a Zariski open subset such that the action of $T$ on itself extends to an action on $V$. Let $N = \Hom(\bC^*, T) \cong \bZ^r$ be the lattice of one-parameter subgroups. Then $V$ is defined by a fan $F \subset N_{\bR} = N \otimes  \bR$, which is a collection of closed rational strictly convex polyhedral cones. 

In \cite{ksz} it is proven that the Chow quotient is itself a toric variety whose fan structure can be defined explicitly. Consider a $k$-parameter subgroup $H$ of the torus $T$ of $V$. Then $H$ corresponds to a sublattice $L \subset N$. We then obtain a natural projection map $p: N \to Q = N /L$ and we can define $G$ to be the projection of our fan $F$ along $p_\bR : N_\bR \to Q_\bR$. This is defined by projecting each cone in $F$ along $p_\bR$ and taking their minimal common refinement. In \cite{ksz} it is shown that the Chow quotient $V \sslash_C H$ is $V(G)$, the toric variety associated to the fan $G$.

More explicitly, we can define $G$ as follows. For each vector $\psi \in N_{\bR}$, define:
$$N(\psi) := \{\sigma \in F | \sigma^{\circ} \cap (L + \psi) \neq \emptyset \}.$$ Vectors $\psi, \psi'$ will be called equivalent if $N(\psi) = N(\psi')$. The closure of each non-empty equivalence class of vectors defines a rational polyhedral convex cone in $N_{\bR}$ invariant under $L$. The collection of images of these cones under the projection map form the fan $G$. 

The Chow quotient is a subspace of the Chow variety, the space of algebraic cycles of a given dimension and homology class. Again, let $H$ be a $k$-parameter subgroup of the larger torus $T$, then we must show that every point of $V \sslash_C H$ corresponds to an algebraic cycle on $V$. For each point $x \in T$ we can define the closure $C_x$ to be $\overline{H\cdot x}$, the orbit of $x$ under the action of $H$. Since the $C_x$ are translates of each other by the action of $T$, each $C_x$ has the same dimension and homology class-- this determines a morphism from $T /H$ to the Chow variety.  One then defines the Chow quotient $V \sslash_C H$ to be the closure of $T/H$ inside the Chow variety. In fact, this construction, which we will sketch below, allows us to determine the toric stack structure $\Chow$. 

\subsection{Construction of the quotient as a toric stack}

Let $\kappa \in G$ be a cone in the quotient fan, and let $e_{\kappa} \in V(G)$ be the associated distinguished point. For $\psi \in \kappa$, Denote by $\calN_0(\kappa)$ the set 
\[ \calN_0(\kappa) = \{ \sigma \in F \mid \sigma^{\circ} \cap (\psi + L)  \textrm{ is  one point } \}. \addtag \]
Now to each distinguished point $e_{\kappa} \in V(G)$, associate a cycle on $V$ of the form: $$E_{\kappa} = \displaystyle\sum_{\sigma \in \calN_0(\kappa)} c(\sigma, L) \cdot \overline{He_{\sigma}}.$$
where we define the multiplicities $c(\sigma,L) \in \bN$ as follows:\\

For each $\sigma \in F$, we denote by $\textrm{Lin($\sigma$)}$, the subspace of $N$ spanned by $\sigma$. Then we define $c(\sigma, L)$ as the index of lattices:
$$ c(\sigma, L) := \bigg[(L + \textrm{Lin($\sigma$)}) \cap N : (L \cap N + \textrm{Lin($\sigma$)} \cap N)\bigg].$$

The cycle over general points is obtained by translating by the action of the torus, thus each point of $V \sslash_C H$ corresponds to an algebraic cycle.\\ 

Let $\kappa \in G$ be a cone in the fan for the Chow quotient described above. Define $$Q_{\kappa} = \bigcap_{\sigma \in \calN_0(\kappa)} p(\sigma \cap N) $$ to be the intersection of the projection of the lattices determined by the cones $\sigma$ such that $e_{\sigma}$ is in the fiber of $e_{\kappa}$. Clearly $Q_{\kappa}$ is a submonoid of $Q$ as $p(\sigma \cap N) \subset p(N) = Q$. \\

In \cite{sam}, the authors prove the following lemma, concluding that the datum above determines a toric stack.  

\begin{lemma} [\cite{sam} Lemma 3] \label{chowtoric} The triple $(G, Q_{\kappa}, Q)$ determines a toric stack, ${\Chow}$. \end{lemma}

\begin{proof} The only thing that needs to be checked is the compatibility condition, i.e. if $\lambda \subset \kappa$ is a face, then $Q_{\kappa}  = \lambda \cap Q_{\kappa}$. Suppose $\sigma \subset F$ is a cone mapping isomorphically to $\kappa$,  then since $p$ is an isomorphism on $\sigma$, it follows that there is a unique face $\tau \subset \sigma$ mapping isomorphically to $\lambda$. Since $\tau \subset \sigma$ and $\tau$ maps isomorphically to $\lambda$, we have that $p(\tau \cap N) \subset p(\sigma \cap N) \cap \lambda$. Furthermore, since $p^{-1}(p(\sigma \cap N) \cap \lambda) \cap \sigma \subset \tau \cap N$, we can conclude that $p(\sigma \cap N) \cap \lambda \subset p(\tau \cap N)$. The result then follows from taking intersections over $\calN_0(\kappa)$ and $\calN_0(\lambda)$. 
\end{proof}

In fact, $\Chow$ is a Deligne-Mumford stack whose underlying coarse moduli space is the Chow quotient $V \sslash_C H$. 

In the following, we discuss some properties of the universal family of the Chow quotient stack, as well as its universal property.  What follows for the remainder of section \ref{3} is somewhat technical and so the proofs can be skipped on a first reading, however we do note that Theoem \ref{univ} (the \emph{universal property}) is the key tool in the proof of our main theorem in section \ref{5}. Finally, we note that subsections 3.5, 3.6, and 3.7 are not necessary for the proof of our main theorem, but contain more properties of the Chow quotient stack that are interesting in their own right.

\subsection{Universal family} 
The Chow stack is naturally a moduli space, as it parametrizes broken orbits of $H$ inside $V$,  and as such, it should carry a universal family. The universal family on the stack $\Chow$ was also constructed in \cite{sam}, in particular,  $\calU \to [V \sslash_C H]$ is the minimal modification of the fan $F$ of $V$ into a toric stack $(F', L_{\kappa'}, L)$ that maps to both $V$ and $[V \sslash_C H]$:

\[
\begin{CD}
\calU @>>>V\\
@VVV  \\
\Chow
\end{CD}
\]\\

Here $L_{\kappa'}$ is defined to be $L_{\kappa'} = L_{\kappa} \cap p^{-1}(Q_\tau)$, where $\kappa'$ is a cone in $F'$.\\

This statement will follow as a corollary to the following lemma:

\begin{lemma} [\cite{sam} Lemma 4] Fix a diagram of morphisms of lattices:

\begin{center}
\begin{tikzcd}
N \ar{r}{id}\ar{d}{p} & N\\
Q
\end{tikzcd}
\end{center}

as well as two fans: $F \subset N$ and $G \subset Q$. Let $\calD$ denote the category of fans $F' \subset N$ that map to both $F$ and $G$ under the given map of lattices. The morphisms of $\calD$ are given by maps of fans $F^{''} \to F'$ that commute with the maps to both $F$ and $G$. Then the category $\calD$ has a terminal object. 
\end{lemma}

\begin{proof}
The terminal object is the collection of cones $p^{-1}(\kappa) \cap \sigma$, where $\kappa$ ranges through all cones in $G$ and $\sigma$ ranges through all cones in $F$. To show that these form a fan, it suffices to show that the intersection of two cones is in the collection and is a face of each.

To show that the intersection of two cones is in the collection, note that $$(p^{-1}(\kappa_1) \cap \sigma_1) \cap (p^{-1}(\kappa_2) \cap \sigma_2) = p^{-1}(\kappa_1 \cap \kappa_2) \cap (\sigma_1 \cap \sigma_2).$$ 

If $x + y \in p^{-1}(\kappa_1) \cap \sigma_1$ is inside the intersection $p^{-1}(\kappa_1 \cap \kappa_2) \cap (\sigma_1 \cap \sigma_2)$,  then by applying $p$ gives that $p(x+y) = p(x) + p(y)$ is in $\kappa_1 \cap \kappa_2$. Therefore we see that $p(x) \in \kappa_1 \cap \kappa_2$ and $p(y) \in \kappa_1 \cap \kappa_2$. Furthermore, $x + y \in \sigma_1 \cap \sigma_2$ and thus $x \in \sigma_1 \cap \sigma_2$ and $y \in \sigma_1 \cap \sigma_2$. Finally, we see that $x, y \in p^{-1}(\kappa_1 \cap \kappa_2) \cap (\sigma_1 \cap \sigma_2)$ and so the second statement also follows. 

\end{proof}

The above Lemma gave rise to a fan $F'$. As $\Chow$ is a toric stack, to obtain $\calU$ we must additionally exhibit its monoid structure: that is, given a cone $\sigma' \in F'$, we must determine $N_{\sigma'}$. This is the minimal choice mapping to $N_{\sigma}$ and $Q_{\tau}$: define $N_{\sigma'} = N_{\sigma} \cap p^{-1}(Q_{\tau})$. The proof of compatibility, given in \cite{sam}, is similar to the proof of Lemma 3.2, and so we omit it here. Thus, we have shown the following:

\begin{lemma} [\cite{sam} Lemma 5] The collection $(F', N_{\sigma'}, N)$ is a toric stack $\calU$. It is the minimal toric stack that maps to both $V$ and $[V \sslash_C H]$. \end{lemma}

We note that the universal family satisfies the following properties:

\begin{prop} The morphism $\calU \to [V \sslash_C H]$ has reduced fibers and is an integral morphism of logarithmic stacks; in particular, it is flat. \end{prop}

\begin{proof} The fact that the morphism has reduced fibers follows directly from \cite{sam} Lemma 6 (see also Lemma 5.2 \cite{ak}). Thus, we only have to demonstrate integrality.\\

To show integrality, it suffices to work locally on $[V \sslash_C H]$, and so the proof reduces to a statement about monoids. Let $N$ be a lattice and let $L$ be a sublattice so that we can define the quotient lattice, $Q = N / L$. Denote the projection map by $p$. Furthermore, suppose $\kappa \in Q $ is a cone, $\sigma \in N$ is a cone such that $p(\sigma) = \kappa$m and assume that for every face $\tau < \sigma$, there exists a face $\lambda < \kappa$ such that $p(\tau) = \kappa$. That is, faces of $\sigma$ map onto faces of $\kappa$. Finally, assume that for each face $\tau$ mapping to a face $\lambda$, we have that $N_{\tau} := N \cap \tau = p^{-1}(Q_{\lambda})$. To prove integrality (and hence flatness), we prove the following lemma:
\begin{lemma} The dual map $\Hom(Q_{\kappa}, \bN) \xrightarrow{p^{\vee}} \Hom(N_{\sigma}, \bN)$ is an injective, integral map of monoids. Therefore, the map $\bZ[Q_{\kappa}] \to \bZ[N_{\sigma}]$ is flat.\end{lemma}
\begin{proof}
To prove injectivity, note that by definition the monoid $N_{\sigma}$ surjects onto $Q_{\kappa}$. We demonstrate integrality by using the equational criterion. If the map was not flat, the locus where flatness fails is closed and torus invariant. Therefore, if this locus is non-empty it must contain a torus fixed point and so it suffices to consider cones $\sigma$ and $\kappa$ of full dimension in their respective lattices. Suppose that $$ p_1 + q_1 = p_2 + q_2$$
\noindent where $p_i \in N^{\vee}_{\sigma}$ and $q_i \in Q_{\kappa}^{\vee}$. Since we assumed that $\sigma$ and $\kappa$ are full dimensional, we have that $N_{\sigma}^{\vee} = \sigma^{\vee} \cap N^{\vee}$ and $Q_{\kappa}^{\vee} = \kappa^{\vee} \cap Q^{\vee}$. We can thus identity $p_i$ and $q_i$ with vectors in the dual spaces of $N$ and $Q$ respectively. We wish to show that $p_1 = w + r_1$ and $p_2 = w + r_2$, where $w \in N_{\sigma}^{\vee}$, $r_i \in Q_{\kappa}^{\vee}$ and that $q_1 + r_1 = q_2 + r_2$. \\

Let $v_1, ..., v_m$ be the extremal rays of $\kappa$, and let $u_k$ denote the lifts of these extremal rays in $\sigma$. Among such $u_k$, choose $u_1, u_2, ..., u_m$ such that $u_i$ maps to $v_i$ and such that $p_1(u_i)$ is minimal among all possible lifts of $v_i$ to an extremal ray of $\sigma$. Notice that $N_{\tau} = Q_{\kappa}$, as the face $\tau < \sigma$ generated by the $u_i$ is a face mapping isomorphically to $\kappa$. Thus, we can identify $\kappa$ with $\tau$ and we will write $p(x)$ for the unique element of $\tau$ mapping to $p(x) \in \kappa$. As every $x \in \sigma$ can be written uniquely as $p(x) + tv$ where $v \in L$, we will make the following definitions:
%
\begin{center} 
$r_1(x) = p_1\big(p(x)\big)$ \\
$w(x) = p_1(tv).$
\end{center}

Then $p_1 = r_1 + w$ and thus we need to show that $w \in N_{\sigma}^{\vee}$ and $r_1 \in Q_{\kappa}^{\vee}$. This is equivalent to showing that $w$ is nonnegative on $\sigma$ and $r_1$ is nonnegative on $\kappa = \tau$. First note that if $x \in L$, then $w(x) = p_1(x) \geq 0$. If $x$ is an extremal ray not in $l$, then $x$ maps to some extremal ray $p(x) \in \kappa$. Since then $x = p(x) = tv$ and $p_1\big(p(x)\big) < p_1(x)$ by our minimality condition,
$$ w(x) = p_1(tv) = p_1(x) - p_1\big(p(x)\big) \geq 0.$$

Since $w \geq 0$ on all extreme rays of $\sigma$, then it is also nonnegative on their convex hull and so $w \geq 0$ on $\sigma$. Thus $ w \in N_{\sigma}^{\vee}$ as claimed. The claim that $r_1 \in Q_{\kappa}^{\vee}$ follows as $p_1$ is already nonnegative on $\sigma$ and thus also on $\tau = \kappa$. 

Repeating this argument for $r_2$ and taking $x = p(x) \in \tau$, we see that:

$$r_1(x) + q_1(x) = p_1(x) + q_1(x) = p_2(x) + q_2(x) = r_2(x) + q_2(x).$$

Thus, $r_q + q_1 = r_2 + q_2$ and the proof is complete.
\end{proof}

This concludes the proof that the universal family has reduced fibers, and is an integral (and thus flat) morphism of logarithmic stacks.
\end{proof}

\subsection{Universal property}

Finally, we introduce the universal property of the Chow quotient stack which first appeared (in less generality) in \cite{sam}. This will be the key tool used in proving the main theorem.

\begin{defn} A toric family of toric $H$-stacks is an $H$-equivariant morphism $X \to S$ of toric stacks which is flat, proper, and equidimensional with reduced fibers. Let $\Lambda$ be the character lattice of $X$ and let $\Pi$ be the character lattice of $S$. Then we assume that $H$ is a subgroup of the torus $T(\Lambda)$ of $X$, as is the kernel of the map of tori: $T(\Lambda) \to T(\Pi)$. Furthermore, a morphism of families between $X \to S$ and $X' \to S'$ is a diagram:

\[
\begin{CD}
X @>>> X'\\
@VVV @VVV\\
S @ >>> S'
\end{CD}
\]

taking $H$ to $H$. \end{defn}

Let $\calC$ denote the category of families of toric $H$-stacks.

\begin{defn} A toric family of toric $H$-stacks mapping to a projective toric variety $V$ is a family of toric $H$-stacks $X \to S$ as in the previous definition together with an $H$-equivariant morphism to $V$. A morphism between families is a morphism in the category $\calC$ which commutes with the maps to $V$. Denote the category of such toric $H$-stacks mapping to $V$ by $\calC_V$. \end{defn}

\begin{theorem}[Universal Property]\label{univ} The Chow quotient stack $\calU \to [V \sslash_C H]$ is the terminal object of $\calC_V$. \end{theorem}

\begin{proof} 
To simplify the exposition, we will first prove that any family of toric $H$-varieties in $\calC_V$ factors through $[ V \sslash_{C} H]$. Such a family is equivalent to a diagram of lattices:

\begin{center}
$\begin{CD}
\Lambda @ >j>> N \\
@V\pi VV\\
\Pi
\end{CD}$
\end{center}

such that $j$ maps  $L \subset \Lambda$, the kernel of $\pi$,  to $L \subset N$ isomorphically,

and also a diagram of fans in the vector spaces spanned by these lattices:

\[
\begin{CD}
\Phi @>j>>F\\
@V\pi VV\\
\Gamma
\end{CD}
\]

As before, let $p: N \to Q$ denote the map of lattices induced by the map $V \to V \sslash_C H$. Since $j$ takes the kernel of $\pi$ to the kernel of $p$, we have an induced map of lattices $i: \Pi \to Q$ giving the following commutative diagram:

\[
\begin{CD}
\Lambda @>j>> N\\
@V\pi VV @VpVV\\
\Pi @>i>> Q
\end{CD}
\]

\begin{lemma} The map $i$ induces a map of fans $i: \Gamma \to G$ from the fan of $S$ to the fan of $[V \sslash_C H]$. \end{lemma}
\begin{proof} We must show that any cone $k \subset \Gamma$ maps into a cone $\kappa \in G$. Let $v_1$ and $v_2$ be two arbitrary vectors in the interior of $\kappa$. As the map $X \to S$ is equidimensional, it takes cones of $\Phi$ onto cones of $\Gamma$ by the criterion of Lemma 4.1 of \cite{karu}. Now let $\{s\}$ be the collection of cones in $\Phi$ with image $k$. Denote by $(v_i)_s$ a preimage of $v_i$ for each $s$. Now consider the collection:

\begin{center}$N\big(i(v_i)\big) = N\big(j(v_i)_s\big) = \{ \sigma \in F \mid p^{-1}(v_i) + L \cap \sigma^o = pt \}$ \end{center}

which first appeared as formula (3.1) in Section 3.2. As $j$ takes $L$ to $L$, then $(v_i)_s + L$ necessarily maps to $j(v_i)_s + L$. Notice that the collection $\{s\}$ covers $(v_i) + L$ since by properness of the maps in the families, any of the preimages of $v_i$ must be the whole $(v_i) + L$. Furthermore, any cone in the preimage maps to the interior of $k$, hence onto $k$ by the equidimensional criterion of Karu (\cite{karu} 4.1). Since $\Phi$ maps to $F$, cones of $\Phi$ must map into cones of $F$ and then for each cone $\sigma \in N\big(i(v_1)\big)$, there is a cone $s$ so that $s^o$ maps to $\sigma^o$. Therefore, $i(v_2) + L \cap \sigma$ is a point and thus $N\big(i(v_1)\big) \subset N\big(i(v_2)\big)$. By symmetry, $N\big(i(v_1)\big) = N\big(i(v_2)\big)$. By definition then $i(v_1)$ and $i(v_2)$ are in the interior of the same cone of the Chow quotient. Calling this cone $\kappa$ we see that $k^o \mapsto \kappa^o$ and thus $k \mapsto \kappa$. This yields the following diagram:

\[
\begin{CD}
\Phi @>j>> F\\
@V \pi VV\\
\Gamma @ >i>> G
\end{CD}
\]

\end{proof}

By the minimality of the universal family we obtain a diagram:

\begin{center}
$\begin{CD}
\Phi @>j'>> F' @>>> F\\
@V \pi VV @VpVV\\
\Gamma @>i>> G\\
\end{CD}$
\end{center}

To conclude, we must show that the submonoids $\Gamma_k = k \cap \Gamma$ factor through the submonoids $G_k$ for any cone $k$ whose interior maps into the interior of $\kappa$. Let $v \in G_{\kappa}$, then since $X \to S$ has reduced fibers, cones in $\Phi$ map onto cones in $\Gamma$. The reduced fibers condition tells us that if we have $s$ mapping to $k$ then the map $\Phi_s \to \Gamma_k$ is surjective. Let $\{s\}$ denote the collection of cones in $\Phi$ mapping to $k$. Choose a list $v_s \in \Phi_s$ of $v$ for each $S$. Define:

$$ N_0(\kappa) = \{ \sigma \in F : w + L \cap \sigma^o = pt \} $$

for all vectors $w$ in the interior of $\kappa$. By definition, $$G_{\kappa} = \displaystyle\bigcap_{\sigma \in N_0(\kappa)} p(N_{\sigma}).$$ Repeating the argument above shows that every cone $\sigma \in N_0(\kappa)$ contains the image of some cone $s$, and therefore $j(\Lambda_s) \subset N_{\sigma}$. Therefore, $i(v) = p\big(j'(v_s)\big)$ for all $s$ is in $G_{\kappa}$. \\

We have given the proof in the case of a toric $H$-variety, that is, in the case where all monoids $\Gamma_k = \Gamma \cap k$. Observe however, that for the general case of a toric stack, we have $\Gamma_k \subset \Gamma \cap k$, and hence $\Gamma_k$ a fortiori factors through $G_k$ by the same argument. This concludes the proof of the universal property of the Chow quotient stack.
\end{proof}

\subsection{Geometry of the universal family}\label{6}

We now undertake a more careful study of the universal family $\mathcal{U} \rightarrow [V \sslash_C H]$. Our main goal is to connect the minimality property of the Chow quotient with the \textit{minimality} (or basicness) condition appearing in logarithmic geometry, as seen for instance in the work of Abramovich-Chen \cite{ac}, Gross-Siebert \cite{gs}, and Gillam \cite{danny}. Indeed, such a connection has already been established in \cite{sam} in the case of logarithmic stable maps, that is, when the dimension of $H$ is 1. Here we give an analogous description in the higher dimensional case.

 In the one dimensional case, the cones in the fan of $\mathcal{U}$ can be explicitly described in terms of the cones in the fan of the Chow quotient stack. This is not true in the higher dimensional case, as the cones of $\mathcal{U}$ can essentially be arbitrary; however, the same description that works in the one dimensional case works for cones of $\mathcal{U}$ mapping with relative dimension $1$. To state the result, we need a lemma: 

\begin{lemma} \label{cones}
Let $\sigma$ be a cone in the fan $F'$ of $\calU$ of dimension $k+1$ mapping onto a cone $\tau$ of dimension $k$ in the fan $G$ of $[V \sslash_C H]$. Then $\sigma$ has precisely either one or two faces mapping isomorphically to $\tau$. 
\end{lemma}

\begin{proof} 
Since we assume that $\sigma$ maps with relative dimension $1$, we have 
\begin{align*}
\dim{\textup{span}(\sigma) \cap L}=1 
\end{align*}
\noindent and hence $\textup{span}(\sigma) \cap L = \bR u$ for some vector $u \in H$. The proof thus reduces to the one dimensional case, which is Lemma 8 of \cite{sam}.  In fact, it also follows that if $\sigma$ has two faces mapping isomorphically to $\tau$, then the two preimages of a vector of $\tau$ differ by a multiple of the same vector $u$.
\end{proof}

\noindent Hence, in the second case, it makes sense to define the following map $c: Q_\tau \rightarrow \bN$. For a lattice point $v$ in $\tau$, there are precisely two lifts $v_1,v_2$ in $\sigma$. We define $c(v)$ to be one less than the number of lattice points in $\sigma$ on the ray connecting $v_1$ with $v_2$: 
\begin{align*}
\textrm{card}(\vec{v_1v_2} \cap N_\sigma)-1
\end{align*}  

With this we can state the following lemma:

\begin{lemma}
With notation as above, if $\sigma$ has precisely one face mapping isomorphically to $\tau$, then $N_\sigma \cong Q_\tau \times \bN$. 
If $\sigma$ has precisely two faces mapping isomorphically to $\tau$, then $N_\sigma \cong Q_\tau \times_\bN \bN^2$, where the map $Q_\tau \rightarrow \bN$ is the map $c$ defined above and the map $\bN^2 \rightarrow \bN$ is the addition map: $(a,b) \mapsto a+b$.
\end{lemma}

\begin{proof}
By the same argument as in the previous lemma, the statement reduces to the case when $H$ is one dimensional, which is treated in \cite{sam} Lemma 9. 
\end{proof}

\noindent In fact unraveling the proof in \cite{sam} gives the following description of the cone $\sigma$: in the first case, every point $x \in N_\sigma$ can be uniquely written as $p(x) + nu$, where $p(x)$ denotes (with abuse of notation) the unique lift of $p(x) \in Q_\tau$ in the face $\tau_1 \cong \tau$, where $u \in L$ is a fixed vector. The vector $u$ can be described canonically: it is the primitive vector on the face $\sigma \cap p^{-1}(0)$. In the second case, an $x \in N_\sigma$ can be uniquely written as  \begin{align*}
x = \frac{a}{c\big(p(x)\big)}p_1(x) + \frac{b}{c\big(p(x)\big)}p_2(x)
\end{align*} \noindent with $a+b=c\big(p(x)\big)$, and with $p_1(x),p_2(x)$ the unique lifts of $p(x)$ in the faces $\tau_1,\tau_2$. Note that in this case any two lifts of the same element in $Q_\tau$ also differ by a multiple of a unique vector $u$ in $L$, though now the choice of $u$ is not canonical. However, once we choose an ordering of the two faces $\tau_1,\tau_2$ isomorphic to $\tau$, we can take $u$ to be the primitive vector in the direction $p_2(x)-p_1(x)$. Summarizing:

\begin{defn} \label{defnreldim}
Let $\sigma \in F'$ be a cone mapping to a cone $\tau \in G$ with relative dimension $1$. Let $K \subset L$ be the one dimensional space $\textup{span}(\sigma) \cap L$. We denote by $u_\sigma$ the primitive vector $u \in K \subset L$ in the direction: 
\begin{itemize}
\item[(a)] $x-p(x)$ if $\sigma$ has a unique face mapping to $\tau$ isomorphically; 
\item[(b)] $p_2(x)-p_1(x)$ if $\sigma$ has precisely two faces mapping to $\tau$ isomorphically. 
\end{itemize}
\end{defn}

Geometrically, a cone $\tau'$ mapping to $\tau \in G$ isomorphically corresponds to a generic component of the fiber of $\mathcal{U}$ over the special point $e_\tau$ of $[V \sslash_C H]$, that is, of the identity of the torus of the toric stratum corresponding to $\tau$. A cone $\sigma$ mapping to $\tau$ with relative dimension $1$ is a divisor in this fiber. The content of Lemma \ref{cones} is then that any such toric divisor in the fiber appears either isolated in a generic component, like a marking in Gromov-Witten theory, or as the intersection of precisely two generic components.

It thus makes sense to call two cones $\tau_1,\tau_2$ mapping isomorphically to $\tau$ \textit{adjacent} if they are the faces of the same cone of dimension $\dim{\tau}+1$. Furthermore, we can call two such cones \textit{comparable} if there is a sequence $\rho_0=\tau_1, \rho_1, \cdots,\rho_n=\tau_2$ such that $\rho_i$ and $\rho_{i+1}$ are adjacent.

\begin{lemma} \label{cones2}
Let $\tau_1$ and $\tau_2$ be any two cones in $F'$ mapping isomorphically to $\tau$. Then $\tau_1$ and $\tau_2$ are comparable. 
\end{lemma} 

\begin{proof}
We need to show that given two cones in $F'$, $\tau_1$ and $\tau_2$, there exists a sequence of cones $$\rho_0=\tau_1, \rho_1, \cdots \rho_n=\tau_2,$$ with $\rho_i$ mapping isomorphically to $\tau$, and cones $\sigma_1,\cdots,\sigma_n$ of dimension $\dim{\tau}+1$ so that $\rho_i,\rho_{i+1}$ are faces of $\sigma_{i+1}$. Since the fibers of the fans $F' \rightarrow G$ are isomorphic in the interior of $\tau$, the statement reduces to the following: any two $0$ dimensional strata in the fiber over a point in $\tau^o$ can be connected by a sequence of $1$ dimensional strata. This is clear since the fiber is a polyhedral complex.
\end{proof}

The previous two lemmas, Lemma \ref{cones} and Lemma \ref{cones2} imply the following corollary regarding the fibers of the universal morphism. We will formally introduce the notion of stable toric varieties in the following section, but choose to state this corollary here. Informally, stable toric varieties can be thought of as the union of toric varieties glued together along their toric boundaries.

\begin{cor} Fibers of  the universal morphism $\mathcal{U} \rightarrow [V \sslash_C H]$ are connected stable toric varieties in the sense of Alexeev (see Section \ref{4}). \end{cor}

\subsection{Gromov-Witten Theory}

To draw the connection with Gromov-Witten theory, we must connect the geometry of $[V \sslash_C H]$ with that of $V$, rather than $\mathcal{U}$. 

To be precise, consider the following association between cones of $F'$ and cones of $F$. For a cone $\sigma' \in F'$, we denote by $\iota(\sigma')$ the unique cone $\sigma \in F$ such that $\sigma'^o \subset \sigma^o$. By construction of the universal family, the map $\iota$ has the following explicit description:
every cone $\sigma' \in F'$ is of the form $\sigma \cap p^{-1}(\tau)$ for a cone $\sigma \in F$ and $\tau \in G$. 

Then 
$\iota\big(\sigma \cap p^{-1}(\tau)\big)=\sigma$ and we can further define, in analogy with $N_0(\tau)$: 

\begin{defn}
Let 
\begin{align*}
N_k(\tau)=\left\{ \sigma \in F: \dim{\textup{span}\big(\sigma \cap p^{-1}(v)\big)}=k \right\} 
\end{align*}

\noindent for $v \in \tau^o$, $\tau \in G$, and define \begin{align*}
M_k(\tau) = \left\{ \sigma' \in F': \dim{\textup{span}\big(\sigma' \cap p^{-1}(v)\big)}=k \right\}  \\
= \left\{ \sigma' \in F': p(\sigma')=\tau, \dim{\sigma'}=\dim{\tau}+k \right\}.
\end{align*}
\end{defn} \begin{lemma}
Let $\tau$ be a cone in $G$. The map $\iota$ maps $M_0(\tau)$ into $N_0(\tau)$, and induces a bijection $M_0(\tau) \cong N_0(\tau)$.
\end{lemma} 

\begin{proof}
Let $\sigma' \in M_0(\tau)$ and let $\sigma = \iota(\sigma')$; take a vector $v \in \tau^o$. By assumption $p^{-1}(v) \cap \sigma'^o$ is a unique vector $w$. Therefore $w=p^{-1}(v) \cap p^{-1}(\tau) \cap \sigma^o = p^{-1}(v) \cap \sigma^o$ and thus $\sigma \in N_0(\tau)$. Conversely, given $\sigma \in N_0(\tau)$, it is clear that $\sigma'=\sigma \cap p^{-1}(\tau)$ is inside $\sigma$; hence the map $p$ maps $\sigma'$ injectively into $\tau$; on the other hand, it hits the interior of $\tau$, and hence, since $p$ maps cones of $F'$ onto cones of $G$, it must map $\sigma'$ to $\tau$ isomorphically, and $\sigma'$ is thus in $M_0(\tau)$.  
\end{proof}

In relative dimension $1$, the situation is only slightly more subtle: If $M_1(\tau)' \subset M_1(\tau)$ denotes the subset of cones $\sigma'$ that have precisely two faces mapping to $\tau$ isomorphically, as in Lemma \ref{cones},  we have the following:

\begin{lemma}
The map $\iota$ takes $M_1(\tau)'$ into $N_1(\tau)$. The image of a cone $\sigma'$ in $M_1(\tau)'$ is a cone $\sigma \in N_1(\tau)$ with precisely two faces in $N_0(\tau)$, namely the images of the two faces of $\sigma'$ in $M_0(\tau)$ under $\iota$. 
\end{lemma}

Finally, we also have 

\begin{theorem}
Let $\tau$ be a cone in $[V \sslash_C H]$ and let $\tau_0, \cdots \tau_n$ be the cones in $F$ that are inside $N_0(\tau)$. For a cone $\sigma \in \iota(M_1(\tau)')$, choose an ordering of the first and second face of $\sigma$ in $N_0(\tau)$. Denote by $i(\sigma)$ and $j(\sigma)$ the index of the first and second face of $\sigma$ in $N_0(\tau)$ in the list $\tau_0,\cdots,\tau_n$ respectively. Then the monoid $Q_\tau$ has the description 
\begin{align*}
Q_\tau = \left\{(v_0,\cdots,v_n,m_\sigma: v_{i(\sigma)} - v_{j(\sigma)} = m_\sigma u_\sigma) \right\} \subset \tau_0 \times \cdots \tau_n \times \prod_{M_1(\tau)'} \bZ,
\end{align*}
where $m_\sigma$ denotes the integer satisfying the equality: $v_{i(\sigma)} - v_{j(\sigma)} = m_{\sigma}u_{\sigma}$.
\end{theorem}
\begin{proof}
Let us provisionally denote the monoid $$\left\{(v_0,\cdots,v_n,m_\sigma: v_{i(\sigma)} - v_{j(\sigma)} = m_\sigma u_\sigma) \right\}$$ by $S$, to save notation. A map $Q_\tau \rightarrow S$ is obtained by mapping a vector $v \in Q_\tau$ to the element $\big(v_0,\cdots,v_n,\frac{v_{i(\sigma)}-v_{j(\sigma)}}{u_\sigma}\big)$, where by $\frac{v_{i(\sigma)}-v_{j(\sigma)}}{u_\sigma}$ we mean the multiple $m_\sigma$ such that $v_{i(\sigma)}-v_{j(\sigma)}=m_\sigma u_\sigma$. That this is well defined is equivalent to saying that Definition \ref{defnreldim} is not a nonsensical definition, which follows from Lemma \ref{cones}.

 An inverse map is provided by mapping $(v_0,\cdots,v_n,m_\sigma)$ to the common image $$p(v_0)=\cdots=p(v_n) \in Q.$$ Note that the image of all the $v_i$ is indeed the same: the image of any two adjacent $v_i$ are the same, as then they are in the two faces of a cone in $N_1(\sigma)$, hence differ by an element $u_\sigma$ of $L$; and by Lemma \ref{cones2}, any two cones $\tau_i,\tau_j$ are comparable. The theorem will then follow if we can show that $p(v_0)=\cdots=p(v_n)$ is indeed in $Q_\tau$.

We first show that the image is in $\tau$. Assume that the vector $v_i$ is in the interior of $\tau_i$ for all $i$. To say that $p(v_i)$ is in the interior of $\tau$ is almost tautological: it means that the set $N(v_i)=\left\{\sigma: v_i + L \cap \sigma^o = pt \right\}$ is the collection $\tau_0, \cdots, \tau_n$. Certainly the collection is a subset of $N(v_i)$. On the other hand, this is true for some choice of $v_i$, namely any mapping into $\tau^o$, which exists by definition of $N_0(\tau)$; but the collection $N(v_i)$ is constant as long as $v_i+L$ varies in the interior of all $\tau_i$, which is as desired. To complete the proof we need to show that $p(v_i) \in Q_\tau$; this is true since $v_i \in N_{\tau_i}$ and $Q_\tau = \displaystyle\cap_{\tau_i} p(N_{\tau_i})$  
\end{proof}

\begin{remark} The above theorem shows that the dual cone of the basic monoid is the moduli space of tropical broken toric varieties of the given type. \end{remark}

Before stating the final corollary, we recall the definition of a \textit{minimal} (or basic) logarithmic structure.

\subsection{Minimal logarithmic structures}

\leavevmode\\

One of the major insights of \cite{gs} was the construction of a \textit{basic} or \textit{minimal} logarithmic structure on the fixed base. It is minimal in the following sense: suppose we have a logarithmic map $X \to S = \Spec(k)$ where $X$ is any  logarithmic scheme and $S$ is a point with an arbitrary logarithmic structure. Additionally, suppose that $S^{\textrm{min}}$ is the same point with the minimal logarithmic structure. Then minimality requires the following diagram to be cartesian:

\[
\begin{CD}
X @>>>X^{\textrm{min}}\\
@VVV @VVV \\
S @>>> S^{\textrm{min}}
\end{CD}
\]\\

That is, any logarithmic structure on $X \to S$ is obtained via pullback from the minimal one. In general, existence of these  minimal objects is equivalent to the existence of a logarithmic algebraic stack representing the moduli problem. For a more general approach to minimality see \cite{danny}.

\begin{cor}
Assume the dimension of the torus $H$ is $1$. Then the monoids $Q_\tau$ are minimal in the sense of Gross and Siebert.
\end{cor}

We will now discuss the other side of the story: the moduli space of stable toric varieties.

\section{Moduli Space of Stable Toric Varieties} \label{4}
First we will define the main objects of the moduli space in question-- stable toric varieties. See \cite{ab} and \cite{ale} for both details as well as proofs.
\subsection{Stable Toric Varieties}
\begin{defn}
A polarized toric variety is a pair $(X,L)$ of a normal projective toric variety $X$ (with torus $T$) and an ample line bundle $L$ on $X$. 
\end{defn}

Recall that every line bundle on a toric variety is linearizable, i.e. the $T$-action on $X$ lifts to a $T$-action on any line bundle $L$ on $X$. Furthermore, recall that there is a one-to-one correspondence between polarized toric varieties $(X,L)$ with linearized line bundle $L$ and integral polytopes $P$ with vertices in the dual lattice $M$. Here, one has $\dim X = \dim P$.

 \begin{defn}A variety $X$ is \textit{seminormal} if any finite morphism $f: X' \to X$ which is a bijection is an isomorphism. \end{defn}
 
 For example, a curve is seminormal if and only if it is locally biholomorphic to the union of $n$ coordinate axes in $\bA^n$.

\begin{defn} A polarized stable toric variety is a pair $(X,L)$ of a projective variety with a linearized ample line bundle such that:
\begin{enumerate}
\item $X$ is seminormal
\item $(X_i, L_i = L|_{X_i})$ are polarized toric varieties.
\end{enumerate}
\end{defn}

The connection between toric varieties and combinatorics yields a lattice polytope for each irreducible component $(X_i, L_i)$. The stable toric varieties can be thought of as a seminormal union of toric varieties glued along $T$-invariant subspaces. Combinatorially, this is equivalent to gluing polytopes along faces. 

\begin{figure}[h!]
\centering
\includegraphics[scale=.5]{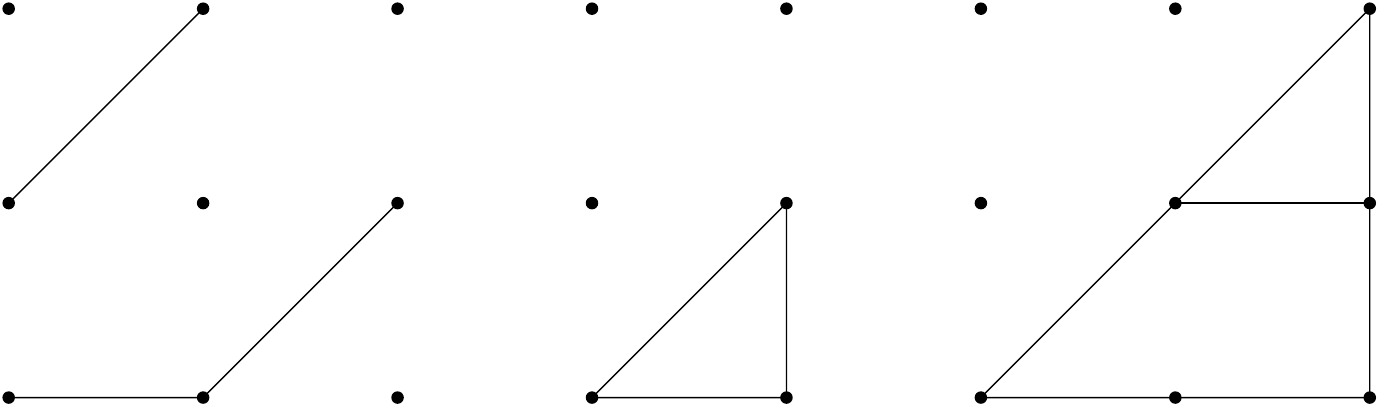}
\caption{Polytopes from left to right: $\bP^1$, $\bP^1 \cup \bP^1$, $\bP^2$, $\bP^2 \cup F_1$}
\end{figure}

\begin{defn} The topological type of a stable toric variety is the topological space $|P| = \cup P_i$, a union of polytopes glued according to $X = \cup X_i$, together with a finite map $\rho: |P| \to M_{\bR}$ such that $\rho|_{P_i}: P_i \to M_{\bR}$ are the embeddings of the lattice polytopes corresponding to $(X_i, L_i)$. \end{defn}

 In addition to considering polarized stable toric varieties, we can define stable toric varieties over projective space, or more generally a projective variety. 

\begin{defn}Let $\bP^n$ be a projective space together with a $T$-linearized sheaf $\calO(1)$, where $T$ is the maximal torus of $\bP^n$. A stable toric variety over $\bP^n$ is a stable toric variety $X$ with a finite morphism $f: X \to \bP^n$ and an isomorphism $L \cong f^*\calO(1)$ of $T$-linearized ample sheaves.  \end{defn}

More generally, one can consider a $T$-invariant subvariety $V \subset \bP^n$ with the sheaf $\calO_V(1) = \calO_{\bP^n}(1)|_V$. This set up naturally allows us to talk about maps (and later logarithmic maps) of stable toric varieties.
 
\begin{theorem} \cite[4.1.1, 4.1.2]{ab} There exists a projective scheme $\Al$ which is a coarse moduli space of stable toric varieties over $V$ with topological type $|Q|$. Furthermore, families of stable toric varieties over $V$ of topological type $|Q|$ form a proper algebraic stack of finite type.  \end{theorem}

\begin{remark} (\cite{icm}) The moduli space $\Al$ contains an open subset which is the moduli space of toric varieties over $V$, and the closure of this open subset is an irreducible component of $\Al$.  In general for some polytopes, Alexeev's moduli space will have several irreducible components. For example, this will occur if one takes a non-convex subdivison of $P$. In what follows, we will refer to $\Alm$ as the main component. As alluded to in the introduction, the logarithmic structure that we add will carve out the main component. \end{remark}

We recall the following:

\begin{lemma}[Lemma 5.3.5 \cite{ab}] The variety $V \sslash_C H$ is the normalization of $\Alm$. \end{lemma}
\begin{remark} In fact, this normalization map above extends to a map on universal families:

\begin{center}
$\begin{CD} 
U @>>> X\\
@VVV @VVV\\
V \sslash_C H @>>> \Alm
\end{CD}$
\end{center}

\end{remark}

In the following section we describe the logarithmic structure we endow this moduli space with.

\subsection{Logarithmic Structure}

We begin by defining logarithmic stable maps of toric varieties prior to constructing the logarithmic stack that parametrizes them. For the remainder of this paper, unless otherwise stated, the notation $X$ will refer to a scheme with a logarithmic structure and $\underline{X}$ will denote the scheme underlying a logarithmic scheme.

\begin{defn} A logarithmic stable map $f$ (of toric varieties) to a fixed projective toric variety $V$ over a logarithmic scheme $S$ is a diagram of logarithmic schemes: 

\begin{center}
$\begin{CD}
X @>f>>V\\
@VVV  \\
S  
\end{CD}$\\

\end{center}

such that $\underline{X} \to \underline{S}$ is a stable toric variety. We require that the map $X \to S$ is flat and logarithmically smooth with fibers reduced stable toric varieties. This is equivalent to an integral and saturated morphism in the category of logarithmic schemes.
\end{defn}

Suppose we have two stable logarithmic maps of toric varieties, $X \to S$ and $Y \to S'$, both with logarithmic maps to $V$ denoted by $f$ and $g$ respectively. Then a morphism of logarithmic stable maps is defined with the following commutative diagram:

\begin{center}
\begin{tikzcd}
X \ar{r}\ar[bend left=40]{rr}{f}\ar{d}&Y \ar{r}{g}\ar{d} & V\\
 S \ar{r}& S'

\end{tikzcd}
\end{center}

\begin{defn} Fix an integral polytope $Q$ and a projective toric variety $V$. Let $\kg$ denote the category fibered in groupoids over $\mathfrak{Logsch}^{fs}$ of stable toric varieties with discrete data defined by $\Gamma = (T, V, Q)$.\end{defn}

 Here, $\mathfrak{Logsch}^{fs}$ denotes the category of fine and saturated logarithmic schemes. More explicitly, it takes a base scheme $S$ with a fine and saturated logarithmic structure $M_S$ and associates to it the set of all families of $T$-equivariant maps of stable toric varieties over $S$ with topological type $Q$ to a fixed target $V$ that are flat and logarithmically smooth with reduced fibers.

We now wish to show that $\kg$ is a logarithmic algebraic stack. This follows from the main theorem of the recent work of Wise \cite{wise}, where the author constructs logarithmic Hom spaces of logarithmic schemes. We summarize the results in the following two statements. 

\begin{theorem}(\cite{wise} Theorem 1.1) \label{wise}
Let $\pi: X \to S$ be a proper, flat, geometrically reduced integral morphism of fine logarithmic algebraic spaces. Let $V$ be a (not necessarily algebraic) logarithmic stack over $S$. Then the morphism

$$\Hom_{\textrm{Logsch}^{fs}_S}(X,V) \to \Hom_{\textrm{LogSch}^{fs}_S} (\underline{X}, \underline{V})$$

is representable by logarithmic algebraic spaces. 
\end{theorem}

\begin{cor}[\cite{wise} Cor 1.1.1]
As in the previous theorem, if additionally $V \to S$ is locally of finite presentation with quasi-compact, quasi-separated diagonal and affine stabilizers, and if $X \to S$ is of finite presentation, then $\Hom_{\textrm{LogSch}^{fs}_S}(X,V)$ is represented by a logarithmic algebraic stack. \end{cor} 

Before proving that the category $\kg$ is a logarithmic stack, we will first describe the logarithmic structures on $X$ and $V$, as the theorem requires as input $X$ and $V$ as logarithmic schemes. Since $V$ is a fixed toric variety, its logarithmic structure is the standard one coming from the toric structure. As $\big(X, \calL \cong f^* \calO_{V}(1)\big)$ is a polarized stable toric variety, we can apply the logarithmic structure defined in 3.1.12 of \cite{Olsson}.

\begin{theorem} The category $\kg$ is a logarithmic algebraic stack. \end{theorem}
\begin{proof}  In \cite{wise}, the author considers $\underline{X}$ as the scheme $X$ with the trivial logarithmic structure. Since a logarithmic map is a map of schemes and a map of logarithmic structures, in this case we see that determining the algebraicity of $\Hom_{\textrm{LogSch}^{fs}_S}(\underline{X},\underline{V})$ is tantamount to showing the algebraicity of $\Hom_{\textrm{Sch}_{\underline{S}}}(\underline{X},\underline{V})$. This is because one has a morphism from the category of $LogSch_S$ to $Sch_{\underline{S}}$ and then $\Hom_{\textrm{LogSch}^{fs}_S}(\underline{X},\underline{V})$ arises as the pullback of $\Hom_{\textrm{Sch}_{\underline{S}}}(\underline{X},\underline{V})$ through this map. Note that $\Al$ is the same as $\Hom_{\textrm{Sch}_{\underline{S}}}(\underline{X}, \underline{V})$, and the $\Hom$ space on the left is nothing other than $\kg$.

As  $\Al$ is an algebraic stack (see \cite{ab} Remark 4.1.2), the previous corollary allows us to conclude that $\kg$ is a logarithmic algebraic stack. \end{proof}

\section{The equivalence} \label{5}

\subsection{Proof of main theorem} 
We begin by discussing the two desired maps between the logarithmic stacks. 

\begin{prop} \label{map1} There exists a map $[ V \sslash_C H ] \to \kg$. \end{prop}
\begin{proof} The existence of this map follows as the Chow quotient stack is itself a family of logarithmic stable toric varieties mapping to $V$. \end{proof}

Before discussing the second desired map between the logarithmic stacks, we show that they are bijective on geometric points.

\begin{lemma} \label{bij} The map of stacks $i: \Chow \to \kg$ is bijective on geometric points. \end{lemma}
\begin{proof}   It suffices to show that every family over $\Spec(\bC)$, that is $f: X \to V$ from a stable toric variety $X$ to a projective toric variety $V$  over $\Spec(\bC)$, comes from the Chow quotient stack.  To show this, we apply the universal property of the Chow quotient stack discussed above.

 In \cite{Olsson} (see in particular, the proof of Lemma 3.7.6), Olsson proves that any stable toric variety $X$ over $\Spec(\bC)$ is pulled back from one of the \textit{standard families}: $$Y \to S = \bA(P)$$ for some $P$. Notice that a logarithmic map $\Spec(\bC) \to \bA(P)$ corresponds to a homomorphism of monoids $P \to R$, where $R$ is the chart defining the logarithmic structure on $\Spec(\bC)$. Therefore the stable toric variety $X$ over $\Spec(\bC)$ is a fiber of the family $\bA(R) \times_{\bA(P)} \bA(R)$ considered as a logarithmic map. 
 
Since $\kg$ is logarithmically smooth (see Lemma \ref{logsmooth}), the map on the fiber $X \to V$ extends \'etale locally to a map around $\Spec(\bC) \in \bA(R)$ and thus we can apply the universal property of the Chow quotient to get that $X \to \Spec(\bC)$ is actually pulled back from the Chow quotient stack. \end{proof}

To show the existence of a map in the other direction we use the universal property of the Chow quotient stack, Theorem \ref{univ}. However, since the universal property requires toric structure, we are first required to prove that $\kg$ is a toric stack.

\begin{theorem} \label{kgtoric} The stack $\kg$ is a toric stack. \end{theorem}

\begin{proof}
We will first show that the stack $\kg$ satisfies the conditions of Theorem \ref{gs}, thus showing that it is a GS toric stack. Then, it follows that $\kg$ is also a toric stack via Remark \ref{2stacks}.

We are first required to exhibit the existence of an action of a torus, $T'$ on $\kg$, and a dense open substack which is $T'$-equivariantly isomorphic to $T'$. First notice that $T$ acts on a map $f: X \to V$ by translating the map $f \mapsto t \cdot f$ where $t \in T$. Now we must show that this action descends to an action of $T'$, that is we need to show that $h\cdot f = f$, where $h \in H$. This requires an isomorphism $V \to V$ commuting with the maps $f: X \to V$ and $h \cdot f: X \to V$. But this map is  multiplication by $h$:

\[
\begin{CD}
X  @>id>> X\\
@VfVV @V Vh\cdot fV\\
V @>h>> V 
\end{CD} 
\]

Recall from Theorem \ref{gs}, that we must show that $\kg$ is normal, has affine diagonal, and that its geometric points have linearly reductive stabilizers.\\

Before proving normality, we recall the following theorem from \cite{ale}:

\begin{theorem}\cite[Corollary 3.1.26]{Olsson}\label{vanish} Let $X$ be a stable toric variety, then $H^p(X, \calO_X) = 0$ for all $p > 0$. \end{theorem}

Now normality follows from proving that $\kg$ is logarithmically smooth:

\begin{lemma} \label{logsmooth} The logarithmic algebraic stack $\kg$ is logarithmically smooth, hence normal. \end{lemma}
\begin{proof}

 The logarithmic deformation theory of a logarithmic map $f:X \rightarrow V$ of two logarithmically smooth schemes is governed by the cohomology groups of the cone of the complex $f^{*}T_V^{\log} \rightarrow T_X^{\log}$. The cohomology long exact sequence for the cone becomes \begin{align*}
0 \rightarrow H^0(f^*T_V^{\log}) \rightarrow H^0(T_X^{\log}) \rightarrow Def(f)
\rightarrow H^1(f^*T_V^{\log}) \rightarrow H^1(T_X^{\log}) \rightarrow Ob(f) \rightarrow \cdots
\end{align*}

\noindent  The logarithmic tangent bundle on a stable toric variety $X$ is $\calO_X^{d}$, where $d = \dim X$. Therefore, by Theorem \ref{vanish}, all higher cohomology of the logarithmic tangent bundle of a stable toric variety vanishes. Since $V$ is a toric variety and $X$ is a stable toric variety, the obstruction group $Ob(f)$ is trivial for every geometric point $f \in \kg$. It follows that the obstruction sheaf on $\kg$ vanishes, and thus $\kg$ is logarithmically smooth.  \end{proof}

Before discussing properties of the diagonal morphism, we must show that adding a logarithmic structure does not alter the automorphisms of the underlying objects. This follows immediately from the representability statement of \cite{wise} mentioned in Theorem \ref{wise}.

To show that the logarithmic algebraic stack $\kg$ has affine diagonal, we will first show that the diagonal is quasi-finite. We will then show properness, which then allows us to conclude that the diagonal map is finite. Finally, note that finite diagonal implies affine diagonal and so this is sufficient.

\begin{lemma} \label{quasifinite} The logarithmic algebraic stack $\kg$ has quasi-finite diagonal. \end{lemma}
\begin{proof}

Consider the following cartesian diagram for any map $g$:

\begin{center}
\begin{tikzcd}
\kg \times_{\kg \times \kg} \Spec(k) \ar{r}\ar{d}{}& \Spec(k) \ar{d}{g}\\
  \kg \ar{r}& \kg \times \kg
\end{tikzcd}
\end{center}

Then quasi-finite diagonal would follow from the fiber product being a finite set. Elements of the fiber product are of the form $( f, p, \phi)$ where $f \in \Al$, $p$ is the point $\Spec(k)$ and $\phi$ is an isomorphism from the pair $(f,f)$ to $(h,k)$ where $(h,k)$ is the image of the point $p$ under the above map $g$. Then this is nothing but a pair of isomorphisms $(f \to h, f \to k)$ and, if we can show there are finitely automorphisms in $\kg$, it would follow that this fiber product is a finite set. 

Since the map $\kg \to \Al$ is representable by Theorem \ref{wise}, and since $\Al$ has finite automorphisms (Remark 4.1.2 (3) of \cite{ab}), it follows that $\kg$ does as well. Consequently, the diagonal of $\kg$ is quasi-finite.

\end{proof}

\begin{lemma} The logarithmic algebraic stack $\kg$ is proper, and therefore has finite diagonal. \end{lemma}
\begin{proof} We verify the valuative criterion. Since $\kg$ has an open dense torus acting on it, it suffices to check on one-parameter subgroups. Take a one-parameter subgroup in $\lambda$ in $\kg$ and consider its lift to $\Chow$ using the map defined above in Proposition \ref{map1}. Since the Chow stack $\Chow$ is proper, there is a unique limit of this lift. We will show that the map $\Chow \to \kg$ is bijective on geometric points, which allows us to conclude that the image of this limit in $\kg$ must also be unique.

Since by Lemma \ref{bij} $i: \Chow \to \kg$ is bijective on geometric points, we conclude that the stack $\kg$ satisfies the valuative criterion of properness.  Properness will follow from demonstrating that the stack $\kg$ is of finite type. 

The stack $\kg$ is locally of finite type by Theorem \ref{wise},  so to show finite type  it suffices to show that $\kg$ is quasi-compact. By \cite[\href{http://stacks.math.columbia.edu/tag/04YA}{Tag 04YA}]{stacks-project} this will follow if we have a surjective morphism $\calV \to \kg$ such that $\calV$ is of finite type. We will choose $\calV$ to be Chow stack, as it is of finite type, and Theorem \ref{bij} shows that we have a surjective morphism on geometric points. By 3.5.3 of \cite{ega},  surjectivity of a morphism on geometric points implies surjectivity everywhere (over $\bC$). 

Thus $\kg$ has finite diagonal as it is proper and has quasi-finite diagonal (Lemma \ref{quasifinite}).  \end{proof}

Finally, we must show that the logarithmic algebraic stack $\Al$ has linearly reductive stabilizers, but this follows (in characteristic 0) since the automorphism groups are finite. Thus, we have shown that $\Al$ is a GS toric stack. Therefore, by Remark \ref{2stacks} the stack $\kg$ is a toric stack. 

\end{proof}

We now show that the universal family $\calX \to \kg$ is also a toric stack. 

\begin{prop} The universal family $\calX$ is a toric stack. \end{prop}
\begin{proof} First, note that by the definition of the moduli problem, the universal projection morphism $\pi: \calX \to \kg$ is flat, logarithmically smooth, proper, and representable. Since we have shown above that $\kg$ is logarithmically smooth and proper with finite diagonal, it follows that $\calX$ is also logarithmically smooth (hence normal) and proper with finite diagonal.

Finally, by the chart criterion of logarithmic smoothness (\cite{kato2} 3.5), the projection $\calX \to \kg$ is locally a toric variety over a toric stack, so locally of global type. Thus, to show that $\calX$ is a toric stack it suffices to show that there is an open dense torus inside $\calX$ acting on $\calX$ whose action extends the action of the torus to itself. We begin by describing the action.\\

Describing an action of a group $G$ on a stack $\mathcal{X}$ is equivalent to describing a $G$ action on $\Hom(S,\mathcal{X})$ for each $S$. Note that a morphism 
\begin{align*}
\xymatrix {S \ar[r] & \mathcal{X} \ar[d]^{\pi} \\ & \kg}
\end{align*} 
is equivalent to the data of the composed map $S \rightarrow \kg$ and a section $$S \rightarrow \mathcal{X} \times_{\kg} S$$ of the projection $\mathcal{X} \times_{\kg} S$. Since a map $S \rightarrow \kg$ is a family of broken toric varieties to $V$, we see that the $S$ points of $\mathcal{X}$ are diagrams

\begin{align*}
\xymatrix{ X \ar[r]^{f} \ar[d]^{p} & V \\ S \ar@/^/[u]^{s}  &}
\end{align*}
 
\noindent as above, where $p \circ s = id$ is a section. We are going to define an action of the torus $T$ of $V$ on such diagrams. Choose an isomorphism $T \cong T/H \times H$, where $H \subset T$ is the original subgroup, and where $T/H$ denotes abusively a complementary subgroup, which will be isomorphic to $T/H$ . We may then uniquely write every $t \in T$ as $t_1h_1$, where $t_1 \in T/H$, $h_1 \in H$. Then $t$ acts on the family $(f,s)$ by $(t_1f,h_1s)$: 

\begin{align*}
\xymatrix{ X \ar[r]^{t_1f} \ar[d]^{p} & V \\ S \ar@/^/[u]^{h_1s}  &}
\end{align*}
 
\noindent Observe that this does not depend on the choice of splitting: If we write $t=t_1h_1=t_2h_2$, we have a commutative diagram

\begin{align*}
\xymatrix{X \ar[r]^{h_2h_1^{-1}} & X \ar[r]^{t_2f} \ar[d]^{p} & V \\S \ar@/^/[u]^{h_1s} \ar[r]_{=}& S \ar@/^/[u]^{h_2s}  &}
\end{align*}

\noindent and thus an isomorphism $(t_1f,h_1s)$ with $(t_2f,h_2s)$. This defines an action of $T$ on $\mathcal{X}$. We now claim that $\mathcal{X}$ contains an open dense torus $T$, and this action restricts to the canonical multiplication action of $T$ on itself. 
Thus $\mathcal{X}$ is a toric stack.  

\begin{lemma}
The torus T is open and dense. 
\end{lemma}  

\begin{proof}
Consider the commutative diagram 
\begin{align*}
\xymatrix {\mathcal{U} \ar[r]^{j} \ar[d] & \mathcal{X} \ar[d] \ar[r]^{f} & V \\ [V //_C H] \ar[r]^{i} & \kg & }
\end{align*}
\noindent Denote the composed map $\mathcal{U} \rightarrow V$ by $g$. We evidently have $T \cong g^{-1}(T)$ by the explicit combinatorial description of the morphism $g$ on the level of fans. But then $$g^{-1}(T) = j^{-1}f^{-1}(T).$$ Since $i$ is bijective, so is $j$, and thus $T$ is identified with $f^{-1}(T) \subset \mathcal{X}$, which is open. Furthermore, note that since $T$ is dense in $\mathcal{U}$, it also is in $\mathcal{X}$. 
\end{proof}

\noindent The torus $f^{-1}(T)$ can be identified with the functor that assigns to a scheme $S$ a diagram 

\begin{align*}
\xymatrix{ X \ar[r]^{f} \ar[d]^{p} & V \\ S \ar@/^/[u]^{s}  &}
\end{align*}
 
\noindent where now each fiber of $X$ is unbroken, and $s$ maps into the locus where the action of $H$ is free - that is, the torus $H$. Such a diagram is identified with an element of $\Hom(S,T)$ by sending a point $p \in S$ to $f \circ s (p) \in T$; Conversely, given a map $m:S \rightarrow T$, we define the family $H \times S \rightarrow V$, mapping each point $(h,p)$ to $h\cdot m(p)$, with the constant section $(1,p)$. 

\begin{lemma}
The action of $T$ on $f^{-1}(T)$ defined above is the natural action of $T$ on itself. 
\end{lemma}

\begin{proof}
Consider a family 
\begin{align*}
\xymatrix{ X \ar[r]^{f} \ar[d]^{p} & V \\ S \ar@/^/[u]^{s}  &}
\end{align*}
which corresponds to the map $S \rightarrow T$ mapping $p$ to $f \circ s (p)$. An element of $t$ acts on $f \circ s (p)$ by sending it to $t\big(f \circ s\big)(p)$. On the other hand, writing $t=t_1h_1$, the action we defined above sends the family to 
\begin{align*}
\xymatrix{ X \ar[r]^{t_1f} \ar[d]^{p} & V \\ S \ar@/^/[u]^{h_1s}  &}
\end{align*}
\noindent which corresponds to the map $S \rightarrow T$ sending $p$ to $\big(t_1f\big)\circ \big(h_1s\bigskip
) (p) = t_1h_1\big(f \circ s\big) (p) = t \big(f \circ s\big) (p)$ since $f$ is by definition $H$ equivariant. 
\end{proof}

Thus, we have demonstrated that $\calX \to \kg$ is also a toric stack, which puts us in a position to use Theorem \ref{univ}

\end{proof}

Combining Theorem \ref{kgtoric} with Theorem \ref{univ} gives us the following theorem:

\begin{prop} \label{map2} There exists a map $\kg \to [V \sslash_C H]$ such that any family in $\kg$ is obtained via pullback through the universal family $\calU \to \Chow$.  We have the following commutative diagram:

\begin{center}
$\begin{CD}
\calX @>>> \calU\\
@VVV @VVV\\
\kg @>>> \Chow
\end{CD}$
\end{center}

\end{prop}
\begin{proof} We have that $\calX \to \kg$ is a toric stack with a map to $V$. Furthermore, both toric stacks $\kg$ and $\Chow$ are equipped with the same underlying torus. Thus, we are able to use the universal property of the Chow quotient stack (Theorem \ref{univ}). This yields a map $\kg \to [V \sslash_C H]$ which factors through the universal family, since by Theorem \ref{univ}, $\Chow$ is the terminal object in the category of toric families of $H$-toric stacks with a map to $V$. 

\end{proof}

The existence of this map above, combined with the existence of the map in Proposition \ref{map1} is enough to conclude our main theorem.

 \begin{theorem} The two toric stacks $\kg$ and $\Chow$ are isomorphic. Furthermore, $\kg$ is a logarithmically smooth, proper, irreducible algebraic stack with finite diagonal. \end{theorem}

\begin{proof} First we note that the two toric stacks and their universal families, $\calX \to \kg$ and $\calU \to \Chow$ have the same underlying tori. This shows that the two stacks have the same lattices, as the lattices are completely determined by the tori underlying the stacks. Furthermore, the maps given in Proposition \ref{map1} and Proposition \ref{map2} yield maps between the fans defining the stacks. Therefore, the conditions of Lemma \ref{keylemma} are satisfied and so the stacks are isomorphic.

\end{proof}

\begin{theorem} The logarithmic algebraic stack $\kg$ is isomorphic to $\Alm$, the normalization of the irreducible main component of $\Al$. \end{theorem}

\begin{proof} Since $\kg$ is isomorphic to the toric stack $\Chow$, the forgetful map $\kg \to \Al$ has finite fibers. As this map is an isomorphism over the (dense) locus of toric varieties, the map is birational. Furthermore, since $\kg$ is normal by Lemma \ref{logsmooth}, and the forgetful map is representable by Theorem \ref{wise}, Zariski's Main Theorem implies that this map must be the normalization map. \end{proof}

\subsection*{Acknowledgments}
The authors would like to thank their advisor Dan Abramovich for his constant support, encouragement, and guidance throughout this project. The authors would also like to thank Jonathan Wise for providing them with an early version of \cite{wise}, Danny Gillam for some helpful discussions, and Anton Geraschenko and Matt Satriano for helpful comments on an earlier version of this paper. Finally, we thank the referees for several comments and suggestions which helped us improve the exposition of this paper.

\bibliography{ascher_molcho_arxiv}
\bibliographystyle{alpha}

\end{document}